\documentclass[12pt, a4paper]{amsart}

\usepackage{amsfonts}
\usepackage{amssymb}
\usepackage{amsmath, amsthm,color}
\usepackage{hyperref}
\usepackage{pdfsync}
\usepackage{bbm, dsfont}
\usepackage[margin=0.96in]{geometry}

\newcommand{\eps}{\varepsilon}

\newcommand{\bR}{\mathbb{R}}

\newtheorem{theorem}{Theorem}[section]

\newtheorem{lemma}[theorem]{Lemma}

\numberwithin{equation}{section}

\begin{document}
\title[On the Sign Distributions of Hilbert Space Frames]{On the Sign Distributions of Hilbert Space Frames}
\author[Nikolai Nikolski, Alexander Volberg]{Nikolai Nikolski, Alexander Volberg}
\thanks{N. Nikolski is partially supported from RNF grant 14-41-00010 and the Chebyshev Lab, SPb University}
\thanks{A. Volberg is partially supported by the NSF DMS-1600065.  }
%for Mathematics, and by the NSF grant DMS-1600065; 
%AV is also supported by the NSF grant DMS-1265549, and VV is also supported by the RFBR grant 14-01-00748.}
\address{Institut de Math\'ematiques de Bordeaux, 
Universit\'e de Bordeaux, Talence, France,
-\, 
and - Chebyshev Laboratory, StPetersburg University}
\email{nikolski@math.u-bordeaux.fr \textrm{(N.\ Nikolski)}}
%\address{Department of Mathematics, Princeton University}
%\email{paata.ivanisvili@princeton.edu}
%\address{Department of Mathematics, Hong Kong University of Science and Technology}
%\email{madli@ust.hk \textrm{(Dong Li)}}

%\address{Department of Mathematics, University of Munich, Germany}
%\email{ \textrm{(R.\ Frank)}}

\address{Department of Mathematics, Michigan State University, East Lansing, MI 48823, USA}
\email{volberg@math.msu.edu \textrm{(A.\ Volberg)}}
\makeatletter
\@namedef{subjclassname@2010}{
  \textup{2010} Mathematics Subject Classification}
\makeatother
\begin{abstract} 
We show that the positive and negative 
parts $ u_{k}^{\pm }$ of any frame in a real $ L^{2}$ space with respect 
to a continuous measure have both ``infinite $ l^{2}$ masses": 1) always, 
$ \sum _{k}u_{k}^{\pm }(x)^{2}=\infty $ almost everywhere (in 
particular, there exist no positive frames, nor Riesz bases), but 2) 
$ \sum _{k=1}^{n}(u_{k}^{+}(x)-u_{k}^{-}(x))^{2}$ can grow ``locally" 
as slow as we wish (for $ n\longrightarrow \infty $), and 3) it can 
happen that $ \sum _{k=1}^{n}u_{k}^{-}(x)^{2}=\, o(\sum _{k=1}^{n}u_{k}^{+}(x)^{2})$, 
and vice versa, as $ n\longrightarrow \infty $ on a set of positive 
measure. Property 1) for the case of an orthonormal basis in $ L^{2}(0,1)$ 
was settled earlier (V. Ya. Kozlov, 1948) using completely different 
(and more involved) arguments. Our elementary treatment includes also 
the case of unconditional bases in a variety of Banach spaces. For 
property 2), we show that, moreover, whatever is a monotone sequence 
$ \epsilon _{k}>0$ satisfying $ \sum _{k}\epsilon ^{2}_{k}=\, \infty 
$ there exists an orthonormal basis $ (u_{k})_{k\, }$in $ L^{2}$ 
such that $ \vert u_{k}(x)\vert \leq \, A(x)\epsilon _{k}$, $ 0<A(x)<\, 
\infty $.
\end{abstract}

\maketitle

\section{The subject. An introduction}

\, \, \rm Let $ (\Omega ,\mu )$ be a measure space ($ \mu $ 
is not a finite sum of atoms), $ L^{2}_{{\Bbb R}}(\Omega ,\mu )$ be
Lebesgue space of real valued functions and $ (u_{k})_{k\geq 1}$ a 
\it frame \rm in $ L^{2}_{{\Bbb R}}(\Omega ,\mu )$. Recall that this 
means that the selfadjoint operator $ S$ (the frame operator),\, 
$$
 Sf= \sum _{k\geq 1}(f,u_{k})u_{k},
$$
\rm is an isomorphism on $ L^{2}_{{\Bbb R}}(\Omega ,\mu )$: there exist 
$ A>0,\, B>0$ such that $ A\cdot I\leq \, S\leq \, B\cdot 
I$, that is\, 
$$
A\Big \Vert f\Big \Vert ^{2}\leq \, \displaystyle \sum _{k\geq 
1}\Big \vert (f,u_{k})\Big \vert ^{2}\leq \, B\Big \Vert f\Big \Vert ^{2}\quad \forall f\in L^{2}_{{\Bbb R}}(\Omega ,\mu )\,.
$$
\, 
\rm The right hand ``half" of this condition is called the ``Bessel sequence 
property"; its dual (equivalent) form is $ \Big \Vert \displaystyle \sum _{k\geq 
1}c_{k}u_{k}\Big \Vert ^{2}\leq $ $ B\displaystyle \sum _{k\geq 1}\Big \vert c_{k}\Big \vert 
^{2}$ for every $ c=\, (c_{k})_{k\geq 1}\in l^{2}$ (look on the 
adjoint $ T^{*}$ to $ Tf=\, ((f,u_{k}))_{k\geq 1}$). Every Riesz 
basis (i.e., an isomorphic image of an orthonormal basis) is a bounded 
frame, and conversely, following the famous Marcus- Spielman- Srivastava 
theorem [MSS2015], every bounded frame is a finite union of Riesz basis 
sequences (i.e., Riesz bases in their closed span).\, 

Below, we consider the question on how can be distributed 
the signs $ sign(u_{k}(x))$ of a frame for $ k=\, 1,2,...$. For 
the case of orthonormal bases $ (u_{k})_{k\geq 1}$ the question was 
raised in [Koz1948]. Kozlov's result is as follows:\, 

\medskip

\, \it Let $ (u_{k})_{k\geq 1}$ be an orthonormal basis in $ L^{2}_{{\Bbb R}}(0,1;dx)$ 
and $ u_{k}^{\pm }(x)=\, max(0,\pm u_{k}(x))$, $ x\in (0,1)$ positive 
and negative parts of $ u_{k}$, respectively. Then $ \sum _{k}u_{k}^{+}(x)^{2}=\, 
\sum _{k}u_{k}^{-}(x)^{2}=\, \infty $ almost everywhere.\, 

\medskip

\rm Kozlov's proof is quite involved and is based on topological properties 
of  Lebesgue measure $ dx$ on $ (0,1)$. In [Koz1948], there are 
also some applications to uniqueness/divergence of Fourier series of 
$ L^{2}$ functions with respect to general orthogonal bases. Later 
on, the same questions were discussed in [Aru1966], [Ovs1980]. We are 
also informed (thanks to D. Yakubovich, University Autonoma de Madrid) 
that the non-existence of positive Riesz bases was requested in the 
perceptive fields theory developed by V. D. Glezer and others, see for 
example [Gle2016]. After this paper appeared in arXiv (1812.06313 in math.FA), Prof. 
A.M.Powell kindly informed us on two more papers \cite{JS2015}  and \cite{PS2016}
where the question on positive bases in $ L_{{\Bbb R}}^{p}(0,1)$ is also 
considered, see comments in 1.2(5) below.\,

\subsection{Results}
\label{res}
We give (simple) proofs to the following theorems.\, 
\begin{theorem}
\label{1.1}
Let $ \mu $ be a continuous measure (i.e., without 
point masses) and $ (u_{k})_{k\geq 1}$ a frame in $ L^{2}_{{\Bbb R}}(\Omega 
,\mu )$. Then\, 
$$
 \sum _{k}(u_{k}^{+}(x))^{2}= \sum _{k}(u_{k}^{-}(x))^{2}=
 \infty, \quad \mu -a.e.
$$
 In particular, there exists no positive frames (nor Riesz bases).\, 
\end{theorem}

\, \, \rm Theorem 1.1 is sharp in several senses: 1) first, 
one \it cannot weaken the frame condition \rm of Theorem 1.1 up to 
``complete Bessel system" condition; 2) secondly, the signs of $ u_{k}(x)$ 
are \it not ``equidistributed" on subsequences \rm of $ (u_{k})$ even 
for orthonormal bases; and 3) third, for sequence spaces $ l$ strictly 
larger than $ l^{2}$, the sequences $ (u_{k}(x))_{k\geq 1}$ can be 
in $ l$ for every $ x\in \Omega $. Precisely, the following facts hold.\, 
\, 
\begin{theorem}
\label{1.2}
 \it Let $ (\Omega ,\mu )$ be a measure space, $ \mu $ 
a continuous measure.\, 

I. There exists a sequence $ (v_{n})_{n\geq 1}$ in $ L^{2}_{{\Bbb R}}(\Omega 
,\mu )$ such that\, 
(1) $ v_{n}\geq 0$ on $ \Omega $,\, 
(2) $ \sum _{n}v_{n}(x)^{2}=\, \infty $ on $ \Omega $,\, 
(3) $ 0<\, \sum _{n}\vert (f,v_{n})\vert ^{2}\leq $ $ B\Vert f\Vert 
^{2},$ $ \forall f\in L^{2}_{{\Bbb R}}(\Omega ,\mu )$, $ f\not= 0$ 
(i.e., $ (v_{n})_{k\geq 1}$ is a complete Bessel sequence).\, 

II. There exists a subset $ E\subset \, \Omega $, $ 0<\mu 
E<\, \infty $, and an orthonormal basis $ (u_{k})_{k\geq 1}$ in 
$ L^{2}_{{\Bbb R}}(\Omega ,\mu )$ such that $ v_{n}:=\, u_{2n}\vert 
E$, $ n=\, 1,2,...$, satisfy conditions (1)-(3) of I \textup($ \Omega $ 
is replaced by $ E$\textup).\, 
\end{theorem}
\begin{theorem}
\label{1.3}
\it Let $ \{b_{n}\},\, b_{n}>0,\, \lim_{n}b_{n}=\, 
\infty $, be a monotone sequence such that
$$
\lim_{n}\frac{ b_{n}}{b_{n-1}}
=\, 1,
$$
 and 
 $$
   \sum \frac{1}{b_{n}}
=\, \infty .
$$
\it Then there exists a weight $ w(x)>0$ on the real line $ {\Bbb R}$ 
such that the orthonormal polynomials $ p_{n}$, $ n=0,1,...$, form 
a basis in $ L^{2}({\Bbb R},wdx)$ and
$$
\Big \vert p_{n}(x)\Big \vert ^{2}\leq \, \frac{\displaystyle C(x)}{\displaystyle 
b_{n}}
$$
\it where $ C(x)>0$ is locally bounded on $ {\Bbb R}$. Notice that 
$ \Big \vert p_{n}(x)\Big \vert ^{2}=\, (p_{n}^{+}(x)\pm p_{n}^{-}(x))^{2}=\, 
p_{n}^{+}(x)^{2}+\, p_{n}^{-}(x)^{2}$.\, 
\end{theorem}
The proof of Theorem 1.3 is given in the 
spirit of the spectral theory of Jacobi matrices, and heavily depends 
on methods developed by A. M\'at\'e and P. Nevai [MaN1983] and R. Szwarc 
[Szw2003], see more references and comments in Section 3 below. \, 

\bigskip

\subsection{Comments}
\label{1.4}
 \bf (1) \rm For measures \it with point masses\rm , 
no analog of Theorem 1.1 can be valid: there exist even orthogonal 
bases of nonnegative functions, for example, the natural basis in $ 
l^{2}=\, L^{2}({\Bbb N},count)$.\,

\, \bf (2) \rm Also, in Theorem 1.1, \it the completeness property 
is essential\rm , i.e. just for Riesz (or even orthonormal) sequences, 
nothing similar is true: the sequences $ (u_{k}^{\pm }(x))_{k\geq 1}$ 
can even have finitely many non-zero coordinates only. Theorem 1.2 
shows that keeping only ``a half of \it frame \rm conditions", namely 
that of complete Bessel systems, we loose the conclusion of 1.1: there 
exist \it positive \rm complete Bessel sequences $ (u_{k})$ for which 
$ \sum _{k}u_{k}(x)^{2}=\, \infty $ a.e.\,

\, \bf (3) \rm Theorem 1.2 implies also a kind of ``non-equidistribution" 
of the signs in the family $ (u_{k})_{k}$ forming a frame (and even 
an orthonormal basis); see comments in Section 4. \,

\, \bf (4) \it The sharpness of Theorem 1.1\rm , as stated in 
Theorem 1.3, implies in particular, that taking $ b_{n}=\, n$ 
we obtain an orthonormal polynomial basis $ (u_{k})_{k}$ in a weighted 
spaces $ L^{2}({\Bbb R},w(x)dx)$, $ w(x)>0$, with the property $ \vert u_{k}(x)\vert 
\leq \, {\frac{c(x)}{k^{1/2}}} $ for every $ x\in {\Bbb R}$, and 
hence\, 
$$
 \sum _{k}\vert u_{k}(x)\vert ^{2+\epsilon }<\, \infty  \forall \epsilon >0\quad \forall x\in {\Bbb R}\,.
$$
\, 
\, \, \rm It is curious that it seems there exist \it no 
classical \rm (or ``semi-classical") orthonormal polynomials which show 
such kind asymptotic behavior. Indeed, in the classical setting, 
the best known estimates are shown by Laguerre orthonormal polynomial 
basis $ L_{k}$, $ k=\, 0,1,...$, in $ L^{2}_{{\Bbb R}}(0,\infty ;e^{-x}dx)$, 
where we have
$$
 L_{k}(x)=\, {\frac{x^{-1/4}e^{x/2}}{\sqrt{\pi } 
k^{1/4}}} Cos(2\sqrt{kx} -\, {\frac{\pi }{4}} )+O(k^{-3/4})\quad x>0\,,
$$
\rm see [Sz1975], p.198), and hence $ \sum _{k}\vert L_{k}(x)\vert ^{4+\epsilon 
}<\, \infty $ ($ \forall \epsilon >0$) almost everywhere, but 
$ (L_{k}(x))_{k\geq 0}\not\in l^{4}$. Similar property holds for Hermite 
normalized polynomials in $ L^{2}_{{\Bbb R}}({\Bbb R};e^{-x^{2}}dx)$.\,

\, \, \bf (5) \rm The theme of the sign distribution of bases 
was developed at least in two other papers, [Aru1966] and [Ovs1980]. 
In [Aru1966], it is proved that for an \it unconditional basis $ (u_{k})_{k\geq 
1}$ \rm in $ L_{{\Bbb R}}^{p}(0,1)$, $ \sum _{k}u_{k}^{\pm }(x)^{p'}=$ 
$ \infty $ a. e. if $ 2\leq p<\infty ,\, {\frac{1}{p'}} +{\frac{1}{p}} 
=\, 1$ (which contains Kozlov's theorem), and $ \sum _{k}u_{k}^{\pm 
}=$ $ \infty $ a.e. if $ 1<p<2$. (We will see in Section 2 that our elementary method entails 
these results and gives more). In [Ovs1980], a stronger 
property is proved under different hypotheses: if a sequence $ (u_{k})\subset 
\, L_{{\Bbb R}}^{2}(0,1)$ is normalized $ \Vert u_{k}\Vert _{2}=\, 
1$, weakly tends to $ 0$ and $ \lim_{n}\int _{E}\vert u_{n}\vert 
dx>\, 0$ for every $ E\subset \, (0,1),\, \vert E\vert >\, 
0$, then $ \sum _{k}u_{k}^{\pm }(x)^{p}=\, $ $ \infty $ a.e. on 
$ (0,1)$, $ \forall p<\, \infty $. Below, we show on a very simple 
example that, there exist positive uniformly minimal complete normalized 
sequences $ (u_{k})\subset $ $ L_{{\Bbb R}}^{2}(0,1)$, $ u_{k}\geq 0$. In \cite{PS2016}, it is shown that there exists neither positive unconditional basis in $ L_{{\Bbb R}}^{p}(0,1)$, $ 1\leq p<\infty $ (already known from 
(Aru1966)), nor positive quasibasis; there are however positive Markushevich bases (minimal complete sequences having complete biorthogonal). In \cite{JS2015} a positive Schauder basis in $ L_{{\Bbb R}}^{1}(0,1)$ is 
 constructed.\, 

\bigskip

 The rest of the paper is as follows: {\S}2 - proof 
of theorem 1.1 and unconditional bases in Banach spaces, {\S}3 - proof 
of theorem 1.3, and possible nonsymmetry between $ u^{\pm }_{k}$, {\S}4 
- proof of theorem 1.2.\,

\bigskip

\, \, \bf Acknowledgements. \rm The first author is highly 
grateful to Sasha and Olga Volberg, as well as to the Math Department 
of the MSU, organizing his short visit to Lansing-Ann Arbor (Fall 2018) 
with remarkable working conditions. He also recognizes a support from 
RNF grant 14-41-00010 and the Chebyshev Lab, SPb University.\, 
The second author is supported by NSF grant DMS 1600065. 
Both authors are grateful to Alexander Powell who indicated to them the papers \cite{JS2015} and \cite{PS2016}.\, 
\, 
\section{Proof of Theorem 1.1, and signs of unconditional 
bases}
\, 
\, \, \rm We start with a simplest version of our principal 
observation.\, 
\, 
\subsection{There exist no nonnegative Riesz bases in $ L^{2}$}
\label{2.1}

This result is not new, see \cite{Aru1966}, \cite{PS2016}. However, seems that our proof is somewhat simpler.
\begin{proof}
\, \, \, \rm Indeed, let $ L^{2}=\, L^{2}_{{\Bbb R}}(\Omega 
,\mu )$, $ \mu $ continuous, $ \mu \Omega <\, \infty $, and assume 
that $ (u_{k})$ is\, 
$$
 \text{a normalized unconditional (= Riesz) basis having }
 u_{k}\geq 0 \,\,\text{on} \,\, \Omega 
$$
\rm and $ f\in L^{2}_{{\Bbb R}}(\Omega ,\mu )$. Using 
the development $ f=\, \sum _{k\geq 1}(f,u'_{k})u_{k}$ (where 
$ (u'_{k})$ stands for the dual sequence, $ (u_{k},u'_{j})= \delta 
_{kj}$), define $ R_{N}f=\, $ $ \sum _{k\geq N}(f,u'_{k})u_{k}$ 
and observe that\, 
$$
 \Big \Vert R_{N}f\Big \Vert _{L^{1}}\leq \, \displaystyle \int 
_{\Omega }\displaystyle \sum _{k\geq N}\Big \vert (f,u'_{k})\Big \vert u_{k}=\, 
\displaystyle \sum _{k\geq N}\Big \vert (f,u'_{k})\Big \vert (u_{k},1)_{L^{2}}=
$$
$$
 (f_{*},R_{N}^{*}1)_{L^{2}}, \rm\text{where}\,\,  f_{*}=\, 
\displaystyle \sum _{k\geq 1}\Big \vert (f,u'_{k})\Big \vert u_{k}\,.
$$
\rm Since $ \Vert f_{*}\Vert _{2}\leq \, B\Vert f\Vert _{2}$, 
it means $ \Big \Vert R_{N}:L^{2}\longrightarrow L^{1}\Big \Vert _{L^{1}}\leq 
\, B\Vert R^{*}_{N}1\Vert _{2}$. But $ \lim_{N}\Vert R^{*}_{N}1\Vert 
_{2}=\, 0$, and the map $ S_{N}f=\, f-R_{N}f$ has a finite 
rank, so we get that $ id:L^{2}\longrightarrow L^{1}$ is compact, which 
is not the case (for example, if $ \mu \Omega =\, 1$, there exists 
a unimodular orthonormal sequences in $ L^{2}$). 
\end{proof} 

\bigskip

\subsection{Remarks on other spaces}
\label{Lp}

\
\bigskip

Let $ L^{p}= L^{p}_{{\Bbb R}}(\Omega 
,\mu )$\rm , $ \mu $ continuous, $ \mu \Omega <$ $ \infty $.

 \bf (1) 
\rm Exactly the same lines (with $ \Vert f_{*}\Vert _{2}$ replaced 
by $ \Vert f_{*}\Vert _{X}$, and $ \Vert R^{*}_{N}1\Vert _{2}$ by $ 
\Vert R^{*}_{N}1\Vert _{X^{*}}$) show that\, 
\, 
\, \it there is no nonnegative unconditional bases in any reflexive 
Banach space $ X$ of measurable functions\, 
\, 
\rm such that $ L^{\infty }(\mu )\subset \, X^{*}\subset \, L^{1}(\mu 
)$, $ X^{*}$ stands for the dual space with respect to the duality 
$ (f,h)=\, \int _{\Omega }f\overline{h}d\mu $.\, 
\, \, \it Example: $ X=\, L^{p}_{{\Bbb R}}(\Omega ,\mu 
)$\rm , $ 1<p<\infty $.\, 
\, \, Later on, we return to $ L^{p}$ spaces in order to 
consider the sign distributions of unconditional bases in more details 
(see point 2.5 below).\,

\, \, \bf (2) \rm One can slightly strengthen property 2.1 
replacing the condition $ u_{k}(x)\geq 0$ a.e. for $ max_{j}h_{j}(x)u_{k}(x)\geq 
0$ a.e. ($ \forall k$) where $ \{h_{j}\}$ stands for a finite family 
of functions taking values $ \pm 1$.\,

 Now, we turn to theorem 1.1 whose proof depends on 
the following elementary lemma and some easy properties of compact 
operators.\, 

\bigskip

\subsection{The tale of two lemmas}
\label{2lemmas}

\begin{lemma}
\label{2.2}
 \it Let $ L^{2}_{{\Bbb R}}(\Omega ,\mu )$ as before, 
$ E\subset \Omega $ with $ 0<\, \mu E<\, \infty $, and $ (v_{k})_{k\geq 
1}$ a sequence in $ L^{2}_{{\Bbb R}}(\Omega ,\mu )$ such that\, 
$$
\sum _{k\geq 1}\vert v_{k}(x)\vert ^{2}\leq \, M^{2}\,\,\text{
for}\,\, x\in E\,.
$$
\, 
\it Then, (1) the map $ V:f\longmapsto \, ((f,v_{k}))_{k\geq 1}$ 
is compact as $ L^{2}(E,\mu )\longrightarrow \, l^{2}$, and (2) 
the map $ V^{*}:(c_{k})_{k\geq 1}\longmapsto \, \sum _{k\geq 1}c_{k}v_{k}\vert 
E$ is compact $ l^{2}\longrightarrow \, L^{2}(E,\mu )$ as well.\, 
\end{lemma}
\begin{proof} \rm (1) Writing $ V=\, V_{N}+\, V'_{N}$, where\, 
$$
 V_{N}f=\, ((f,v_{1}),...,(f,v_{N}),0,0,...)\,,
$$
\rm we get for every $ c=\, (c_{k})_{k\geq 1}\in l^{2}$, $ f\in L^{2}(E)$ 
and $ N\geq 1$,\, 
$$
\vert (V'_{N}f,c)\vert =\, \vert \sum _{k\geq N}c_{k}(f,v_{k})\vert 
\leq \, \Vert f\Vert _{2}\sum _{k\geq N}\vert c_{k}\vert \cdot \Vert 
v_{k}\vert E\Vert _{2}\leq 
$$
$$
 \Vert f\Vert _{2}\Vert c\Vert _{2}(\sum 
_{k\geq N}\Vert v_{k}\vert E\Vert ^{2}_{2})^{1/2}=:\Vert f\Vert _{2}\Vert 
c\Vert _{2} \cdot \epsilon _{N}\,.
$$
\rm Hence $ \Vert V'_{N}\Vert \leq \, \epsilon _{N}$, where $ \epsilon 
_{N}\longrightarrow 0$ since $ \epsilon _{1}\leq \, M(\mu E)^{1/2}<\, 
\infty $\it . \rm The claim follows.\, 
\, \, (2) $ V^{*}$ is the adjoint of $ V$ of point (1). 
\end{proof}
\, 
\begin{lemma}
\label{2.3}
 \it Let $ (v_{k})_{k\geq 1}$ and $ E\subset \Omega $ 
be as in Lemma \ref{2.2} and $ (u_{k})_{k\geq 1}$ a frame in $ L^{2}_{{\Bbb R}}(\Omega 
,\mu )$. Then, the operators\, 
$$
 Uf=\, \sum _{k\geq 1}(f,v_{k})u_{k}\,\, \text{acting as }\,\,
 L^{2}(E,\mu )\longrightarrow L^{2}_{{\Bbb R}}(\Omega ,\mu )
$$
\it its adjoint $ U^{*}f=$ $ \sum _{k\geq 1}(f,u_{k})v_{k}$, and $ U^{'}$, given by
$$
U'f= \sum _{k\geq 1}(f,v_{k})v_{k}\vert E:\, L^{2}(E,\mu 
)\longrightarrow L^{2}_{{\Bbb R}}(E,\mu )
$$
\it are compact.
\end{lemma}
\, 
\begin{proof}\rm For $ U$, the frame definition entails $ \Vert Uf\Vert ^{2}_{2}\leq 
\, B\Vert Vf\Vert ^{2}_{l^{2}}$ for every $ f\in L^{2}(E,\mu )$, 
and the claim follows from Lemma \ref{2.2}.\, 
\, \, \, For the operator $ U'$, we repeat the estimate 
of Lemma \ref{2.2}:\, 
$$
 \vert \sum _{k\geq N}(f,v_{k})v_{k}\vert ^{2}\leq (\sum 
_{k\geq N}\vert (f,v_{k})\vert ^{2})(\sum _{k\geq N}\vert v_{k}\vert ^{2})\,,
$$
\rm which gives the result after integration over $ E$. 
 \end{proof} 
\, 
\subsection{Proof of the Theorem 1.1}
\label{2.4}
\begin{proof}
 \rm Suppose $ \sum _{k\geq 1}(u_{k}^{-}(x))^{2}<\, 
\infty $ on a set of positive measure. Then there exist $ E\subset \Omega 
$ and $ M>0$ such that\, 
$$
 \sum _{k\geq 1}(u_{k}^{-}(x))^{2}\leq \, M^{2}\quad \forall x\in E\,\, \text{and} \,\, 0<\, \mu E<\, \infty \,.
 $$
\, 
\rm This implies the same contradiction as in point 2.1 that the natural 
embedding $ L^{2}(E,\mu )\, \hookrightarrow \, L^{1}(E,\mu )$ 
is compact. The steps are as follows.

(1) Setting $ v_{k}=\, u^{-}_{k}$, we have from Lemma \ref{2.2}\, 
$$
 \Vert Vf\Vert _{l^{2}}^{2}= \sum _{k\geq 1}\vert (f,u^{-}_{k})\vert 
^{2}\leq \, \Vert V\Vert ^{2}\cdot \Vert f\Vert ^{2}\,\, \text{on}\,\, L^{2}(E,\mu 
)
$$
\rm and from the frame definition $ \sum _{k\geq 1}\vert (f,u_{k})\vert 
^{2}=\, (Sf,f)\leq \, B\Vert f\Vert ^{2}_{2}$ ($ \forall f\in 
L^{2}(\Omega ,\mu )$). Hence\, 
$$
\sum _{k\geq 1}\vert (f,u^{+}_{k})\vert ^{2} \leq 
 C^{2}\Vert f\Vert ^{2}\,\,\text{on}\,\,L^{2}(E,\mu )\,,\,\, C^{2}\leq \, 2(\Vert 
V\Vert ^{2}+B)\,.
$$

\rm (2) It follows from $ u^{+}_{k}=\, u_{k}+\, u^{-}_{k}$\it , 
\rm (1) and Lemma \ref{2.3} that $ W$,\, 
$$
 Wf:=\, \sum _{k\geq 1}(f,u^{+}_{k})u^{-}_{k}\vert E\,,
$$
\rm acting as $ L^{2}(E,\mu )\longrightarrow L^{2}_{{\Bbb R}}(E,\mu )$ 
is compact.\,

(3) Now, the quadratic form $ (Sf,f)=\, (Uf,f)+\, (Wf,f)+\, 
(Xf,f)$ on $ L^{2}(E,\mu )$, where\, 
$$
 Xf:=\, \sum _{k\geq 1}(f,u^{+}_{k})u^{+}_{k}
$$
\rm is equivalent to $ (f,f)=\, \Vert f\Vert ^{2}$ (in the sens 
$ A\Vert f\Vert ^{2}\leq \, (Sf,f)\leq \, B\Vert f\Vert ^{2}$), 
and the forms $ (Uf,f)$ and $ (Wf,f)$ are compact on $ L^{2}(E,\mu )$. 
It implies that $ (Xf,f)$ is equivalent to $ \Vert f\Vert ^{2}$ on 
a subspace $ H\subset \, L^{2}(E,\mu )$ of finite co-dimention.\, 

(4) The latter property means that the compression\, 
$$
X_{E}:L^{2}(E,\mu )\longrightarrow L^{2}(E,\mu ),\quad X_{E}f=\, Xf\vert E, \,\, f\in L^{2}(E,\mu 
$$
\rm is Fredholm. Let $ R:L^{2}(E,\mu )\longrightarrow L^{2}(E,\mu )$ 
be a regularizer of $ X_{E}$, a bounded operator such that\, 
$$
RX_{E}=\, id+\, K\,\, \text{where} \,\,K:L^{2}(E,\mu )\longrightarrow 
L^{2}(E,\mu )\,\, \text{is compact}\,.
$$

\rm (5) Show that $ X_{E}:L^{2}(E,\mu )\longrightarrow \, L^{1}(E,\mu 
)$ is compact. Indeed, similarly to Lemmas \ref{2.2} and \ref{2.3}, the norms $ 
\Vert X_{E,N}:L^{2}(E,\mu )\longrightarrow L^{1}(E,\mu )\Vert $ tends 
to zero as $ N\longrightarrow \infty $, where\, 
$X_{E,N}f=\, \sum _{k\geq N}(f,u^{+}_{k})u^{+}_{k}\vert 
E$. In fact,
\begin{align*} 
& \Vert \sum _{k\geq N}(f,u^{+}_{k})u^{+}_{k}\Vert _{L^{1}  (E,\mu 
)}\leq \, \sum _{k\geq N}\vert (f,u^{+}_{k})\vert \cdot \Vert u^{+}_{k}\Vert 
_{L^{1}  (E,\mu )}=
\\
&\sum _{k\geq N}\vert (f,u^{+}_{k})\vert \cdot 
(u^{+}_{k},1)_{L^{2}  (E,\mu )}\leq 
\\
& \leq \, (\sum _{k\geq N}\vert (f,u^{+}_{k})\vert ^{2})^{1/2}(\sum 
_{k\geq N}\vert (u^{+}_{k},1)_{L^{2}  (E,\mu )}\vert ^{2})^{1/2}\leq 
\\
& C\Vert f\Vert (\sum _{k\geq N}\vert (u^{+}_{k},1)_{L^{2}  (E,\mu 
)}\vert ^{2})^{1/2}:=\, C\Vert f\Vert \epsilon _{N}\,,
\end{align*}
\rm and $ \lim_{N}\epsilon _{N}=\, 0$ in view of (1) above. Now, 
regarding the identity $ RX_{E}=$ $ id+$ $ K$ as acting from $ L^{2}(E,\mu 
)$ to $ L^{1}(E,\mu )$, we get that the natural embedding $ id:L^{2}(E,\mu 
)$ $ \hookrightarrow $ $ L^{1}(E,\mu )$ is compact which contradicts 
to $ \mu E>0$. 
\end{proof} 
\, 
\subsection{Sign distributions for bases in more general spaces}
\label{2.5}
 \rm Here 
we give ``an abstract version" of the reasoning from 2.1-2.4 (without 
trying to find the most general setting). Let as before, $ (\Omega ,\mu 
)$ be a measure space with a continuous measure, and (WLOG) $ \mu \Omega 
<\, \infty $. $ X$ will be a real \it reflexive Banach lattice 
\rm of measurable functions such that\, 
$$
L^{\infty }\subset \, X\subset \, L^{1}\,\,\text{
and} \,\, X^{*}=\, \{h:\, hf\in L^{1},\, \forall f\in X\}
$$
\rm with the duality $ (f,h)=\, \int _{\Omega }f\,hd\mu $.\, 

\medskip

\noindent{\bf Example:} $ X=\, L^{p}_{\bR}(\Omega ,\mu )$\rm , 
$ 1<p<\infty $, or $ X$ is a (rearrangement invariant) symmetric space 
of measurable functions, see [KPS1978].\, 
\, \, Let $ U=\, (u_{k})_{k\geq 1}$ be a normalized unconditional 
basis in $ X$, $ U'=$ $ (u_{k}')_{k\geq 1}$ the dual basis, $ (u_{k},u'_{j})=\, 
\delta _{kj}$, so that\, 
$$
 f=\, \sum _{k\geq 1}(f,u'_{k})u_{k}\,\, \text{ for every }
\,\, f\in X\,.
$$
\rm Denote\, 
$$
 Coef(U)=\, \{c(f):=\, \{(f,u'_{k})\}:\, f\in 
X\}
$$
\rm the sequence space of coefficients (if needed we will add the space to the notation: $Coef(U, X)$); this is a sequence lattice, 
$ (a_{k})\in Coef(U)\, \Rightarrow \, (\lambda _{k}a_{k})\in 
Coef(U),\, \forall (\lambda _{k})\in l^{\infty }$, where the standard 
$ 0-1$ sequences form an unconditional basis. Clearly, $ Coef(U')=\, 
(Coef(U))^{*}$ (with respect to the duality $ (a,b)=\, \sum _{k\geq 
1}a_{k}b_{k}$).\, 
\, \, With this notation, here is our claim on the sign distributions.\, 
\, 
\begin{theorem}
\label{2.6}
 \it Let $ X$ be a reflexive Banach lattice of measurable 
functions satisfying the above conditions, and $ U=\, (u_{k})$ 
be a normalized unconditional basis in $ X$. Then, for every $ E\subset \, \Omega 
$, $ \mu E>\, 0$,\, 
$$
 \Big (\displaystyle \int _{E}u_{k}^{+}d\mu \Big )_{k\geq 
1}\not\in (Coef(U))^{*} \rm \text{and}\,\, \Big (\displaystyle \int _{E}u_{k}^{-}d\mu 
\Big )_{k\geq 1}\not\in (Coef(U))^{*}\,.
$$
\end{theorem}

\begin{proof} Here is the proof of Theorem \ref{2.6}. \rm The reasoning repeats our steps 
above. Namely, let\, 
$$
 V^{\pm }f=\, \sum _{k\geq 1}(f,u'_{k})u^{\pm }_{k},
\,\,\text{so that}\,\, id=\, V^{+}-\, V^{-}
$$
\rm and\, 
$$
V_{N}^{\pm }f= \sum _{k\geq N}(f,u'_{k})u^{\pm }_{k},\,\,
 f\in X,\,\,  N=\, 1,2,...
 $$

\rm Now, assuming $ R^{-}u:=\, \Big (\displaystyle \int _{E}u_{k}^{-}d\mu 
\Big )_{k\geq 1}\in (Coef(U))^{*}$ for some $ E$, $ \mu E>0$, we obtain\, 
$$
 \Vert V_{N}^{-}f\Vert _{1}\leq  \sum _{k\geq N}\vert 
(f,u'_{k})\vert \int _{E}u^{-}_{k}d\mu =\, (c(f_{*}),R_{N}^{-}u),
$$
\rm where $ f_{*}=\, \sum _{k\geq 1}\vert (f,u'_{k})\vert u_{k}$ 
(with $ \Vert f_{*}\Vert _{X}\leq \, C\Vert f\Vert _{X}$, unconditional 
basis) and $ R_{N}^{-}u=\, \{0,...0,\int _{E}u^{-}_{N+1}d\mu ,...\}$. 
Since $ (u'_{k})$ is a basis in $ X^{*}$ (reflexivity of $ X$), we 
get\, 
$$
 \Big \Vert V_{N}^{-}:X\longrightarrow L^{1}\Big \Vert \leq 
\, C\Big \Vert R_{N}^{-}u\Big \Vert _{X^{*\, }}\longrightarrow \, 
0\,\,\text{as}\,\,N\longrightarrow \infty \,.
$$
\, 
\rm The same for $ V_{N}^{+}$:\, 
$$
 \Vert V_{N}^{+}f\Vert _{1}\leq \sum _{k\geq N}\vert 
(f,u'_{k})\vert \int _{E}u^{+}_{k}d\mu = (c(f_{*}),R_{N}u-R_{N}^{-}u)\,,
$$
\rm where $ R_{N}u=$ $ \{0,...0,\int _{E}u_{N+1}d\mu ,...\}$, and hence\, 
\, 
 $$
  \Vert V_{N}^{+}:X\longrightarrow L^{1}\Vert \leq  C\Vert R_{N}u-R_{N}^{-}u\Vert _{X^{*}}\leq \, C(\Vert R_{N}u\Vert 
_{X^{*\, }}+\, \Vert R_{N}^{-}u\Vert _{X^{*}})\,.
$$ 
\rm But $ \displaystyle \int _{E}u_{k}d\mu =\, (\chi _{E},u_{k})$ 
and hence $ \Big (\displaystyle \int _{E}u_{k}d\mu \Big )_{k\geq 1}\in 
(Coef(U))^{*}$ ($ X\subset \, L^{1}$, and so $ L^{\infty }\subset \, 
X^{*}$). It implies $ \lim_{N}\Vert R_{N}u\Vert _{X^{*}}=\, 0$, 
and as above, we conclude that both $ V^{+}$, $ V^{-}:\, X\longrightarrow \, 
L^{1}$ are compact operators and $ id=\, V^{+}-\, V^{-}$. 
But there exists a unimodular sequence $ v_{n}$ in $ X$ which tends 
weakly to zero (it is clear when replacing $ (\Omega ,\mu )$ by isomorphic 
measure space $ ((0,1,dx)$), but $ \Vert v_{n}\Vert _{1}=\, \mu \Omega 
>0$. Contradiction. 
 \end{proof} 

\bigskip

\subsection{
Now we give an application of Theorem \ref{2.6}}
\label{appl}

\medskip

\bf (1) Type, cotype, and unconditional bases. \rm Recall (see for 
example [Woj1996], point III.A.17) that a Banach space $ X$ is said 
to \it have cotype $ q,\,\, 2\leq q\leq \infty $\rm , if for some 
constant $ C>0$ and for every finite sequence $ x=\,\, (x_{j})$, 
$ x_{j}\in X$,\ \par 
$$
 C\displaystyle \int _{0}^{1}\Big \Vert \displaystyle \sum _{j}r_{j}(t)x_{j}\Big \Vert 
dt\geq \,\, \Big \Vert x\Big \Vert _{l^{q}}:=\,\, \Big (\displaystyle \sum _{j}\Big \Vert 
x_{j}\Big \Vert ^{q}\Big )^{1/q}\,,
$$
and it \it has type $ q,\,\, 1\leq q\leq 2$\rm , if\ \par 
$$
 \displaystyle \int _{0}^{1}\Big \Vert \displaystyle \sum _{j}r_{j}(t)x_{j}\Big \Vert 
dt\leq \,\, C\Big \Vert x\Big \Vert _{l^{q}}\,,
$$
 where $ (r_{j})_{j\geq 1}$ stands for the sequence of Rademacher 
functions. It is known (and is proved in [Woj1996], Ch. III.A) that 
$ X$ has type $ q$ if and only if $ X^{*}$ has cotype $ q'$, $ {\frac{1}{q'}} 
+\,\, {\frac{1}{q}} =\,\, 1$, and \it if $ X$ has type $ q_{1}\leq 
2$ and a cotype $ q_{2}\geq 2$ and if $ U=\,\, (u_{k})$ is a normalized 
unconditional basis in $ X$ then\ \par 
\ \par 
  \centerline{ $ l^{q_{1}}\subset \,\, Coef(U,X)\subset \,\, l^{q_{2}}$\rm .}
\ \par 
\bf Corollary. \it If in condition of Theorem 2.3, the lattice $ X$ 
has a cotype $ q_{2}$ then\ \par 
\ \par 
  \centerline{ $ \Big (\displaystyle \int _{E}u_{k}^{\pm }d\mu \Big )_{k\geq 
1}\not\in l^{q_{2}  '}$ \rm ($ \forall E\subset $ $ \Omega ,$ $ \mu E>0$), 
whence $ \sum _{k\geq 1}(u_{k}^{\pm }(x))^{q_{2}  '}=\,\, \infty $ 
a.e. $ \Omega $.}

\medskip

\rm Indeed, $ l^{q'_{2}}\subset \,\, Coef(U,X)^{*}$, 
and the first claim follows from the theorem. Also 
$$ \Big (\displaystyle \int _{E}u_{k}^{\pm 
}d\mu \Big )^{q_{2}  '}\leq \,\, c\displaystyle \int _{E}(u_{k}^{\pm })^{q_{2}  '}d\mu \,,
$$
whence $ \displaystyle \int _{E}\displaystyle \sum _{k}(u_{k}^{\pm })^{q_{2}  '}d\mu 
=\,\, \infty $ for every $ E,\,\, \mu E>0$, which is equivalent 
to
$$
 \sum _{k\geq 1}(u_{k}^{\pm }(x))^{q_{2}  '}= \infty \quad
a.e. \,\, \Omega \,.
$$ 
\ \par 
\bf (2) The spaces $ X=$ $ L^{p}_{{\Bbb R}}(\Omega ,\mu )$. \rm It 
is known (and is basically equivalent to   Khintchin's inequality, see 
[Woj1996], point III.A.22) that $ L^{p}$ \it is of type $ q_{1}=\,\, min(2,p)$ 
and of cotype $ q_{2}=\,\, max(2,p)$\rm , and hence\ \par 
\ \par 
  \centerline{ $ Coef(U,L^{p})\subset $ $ l^{q}$, where $ q=\,\, max(2,p)$.}
\ \par 
\rm (It is curious to note how different is the coefficient space for 
the standard trigonometric \it Schauder basis \rm of $ L^{p}(0,2\pi $): 
the Hausdorff--Young inequality tells that $ Coef(e^{inx},L^{p})\subset $ 
$ l^{p'}$ for $ 1<p\leq 2$ and $ Coef(e^{inx},L^{p})\subset $ $ l^{2}$ 
for $ p\geq 2$).\ \par 
\ \par 
\bf Corollary. \it Let $ X=$ $ L^{p}_{{\Bbb R}}(\Omega ,\mu )$, $ 1<p<\infty 
$, and $ U=\,\, (u_{k})$ a normalized unconditional basis in $ X$. 
Then for every $ E\subset \,\, \Omega ,\,\, \mu E>0$, we have\ \par 
\ \par 
  \centerline{ for $ 1<p\leq 2$, $ \Big (\displaystyle \int _{E}u_{k}^{\pm 
}d\mu \Big )_{k\geq 1}\not\in l^{2}$, and in particular $ \sum _{k\geq 
1}u_{k}^{\pm }(x)^{2}=\infty $ a.e.,}
\ \par 
  \centerline{ \it for $ 2<p<\infty $, $ \Big (\displaystyle \int _{E}u_{k}^{\pm 
}d\mu \Big )_{k\geq 1}\not\in l^{p'}$, and in particular $ \sum _{k\geq 
1}u_{k}^{\pm }(x)^{p'}=\infty $ a.e.,}
\ \par 
\it where $ {\frac{1}{p'}} +\,\, {\frac{1}{p}} =\,\, 1$\rm . 
$   $\ \par 

\medskip

 The necessary condition $ (u_{k}^{\pm }(x))_{k\geq 1}\not\in l^{p'}$ 
a.e. for $ p\geq 2$, as well as a weaker condition $ (u_{k}^{\pm }(x))_{k\geq 
1}\not\in l^{1}$ a.e. for $ 1<p<2$, were found already in [Aru1966].

 %%%%%%%%%%%%%%%%%%%%%%%%

\section{Pointwise behavior of orthogonal polynomials, 
and proof of Theorem 1.3.}
\label{3}

\, \rm Here we show that the exponent 2 in Theorem 1.1 cannot 
be improved: for every $ \epsilon _{k}\searrow 0$ having $ \sum _{k}\epsilon 
_{k}^{2}=\, \infty $, there exists an orthonormal basis $ (u_{k})$ 
with $ \vert u_{k}(x)\vert \leq \, C(x)\epsilon _{k}$ a.e.; in 
particular, taking $ \epsilon _{k}=\, (k+1)^{-1}$, we get $ \sum _{k\geq 
1}\vert u_{k}(x)\vert ^{2+\epsilon }<\, \infty $ a.e ($ \forall \epsilon 
>0$). Theorem 1.3 is a simple restating of Theorem 3.1 below. The proof 
of Theorem 3.1 is based on the three terms recurrence for orthogonal 
polynomials but its direct application (replacing moduli of sums by 
sums of moduli) fails. Instead, we use a subtle reasoning introduced 
in a similar situation in important papers by A. M\'at\'e and P. Nevai 
[MaN1983] and R.Szwarc [Szw2003]. The basic facts of the theory of 
orthogonal polynomials are contained (for example) in the books [Sz1975], 
[Ber1968], [Sim2005]. One of them, the classical J. Favard theorem (1935), 
claims that whatever are real sequences $ b_{k}\in {\Bbb R}$ and $ a_{k}>0$ 
and the sequence of polynomials $ p_{k}$, $ deg(p_{k})=\, k$, 
$ k=0,1,...$ defined by the three term recurrence 
$$
 xp_{k}(x)=\, a_{k+1}p_{k+1}(x)+\, b_{k}p_{k}(x)+\, 
a_{k}p_{k-1}(x), \quad k=\, 0,1,2,...\,,
$$
\rm where $ p_{0}=\, 1$, $ p_{-1}(x)=\, 0$, there exists 
(at least one) Borel measure $ \mu \geq 0$ on the real line such that 
$ p_{k}\in L^{2}(\mu )$ ($ \forall k\geq 0$) and $ (p_{k},p_{j})_{L^{2}  (\mu 
)}=\, \delta _{k,j}$ (Kronecker delta). 

In fact, the measure $\mu$ is the scalar spectral measure of the associated tridiagonal (selfadjoint) Jacobi matrix $J$  having $(b_k)_{k\ge 0}$ on the main diagonal and $(a_k)_{k\ge 1}$ on two side diagonals.

Another classical theorem 
(T. Carleman) tell us that such a measure is unique if $ \sum _{k\geq 0}{\frac{1}{a_{k}}} 
=\, \infty $ (the so-called ``determined case") - the condition 
is obviously satisfied in case of Theorem 3.1 below. It follows that 
the polynomials are dense in $ L^{2}(\mu )$, and hence $ (p_{k})_{k\geq 
0}$ forms an orthonormal basis in $ L^{2}(\mu )$. A huge theory of orthogonal polynomials and the associated Jacobi matrices is (partially) presented
in books mentioned above.

\bigskip

%%%%%%%%%%%%%%%%%%
We use here the work of R. Szwarc \cite{Szw2003}. We just repeat several calculations from this article to get the following result.
\begin{theorem}
\label{bn}
Let $\{b_n\}$ , $b_n>0$, $b_n\to \infty$, be a monotone sequences such that $b_n/b_{n-1} \to 1$, and  $\sum b_n^{-1} =\infty$ and let $a_n$ be such that
$a_n = \frac1{2B}  \sqrt{b_n b_{n-1}}$, where $0<B<1$. Then the Jacobi matrix with $\{b_n\}$ on the main diagonal and $\{a_n\}$ on two other diagonals will have 
absolutely continuous spectrum and  the orthogonal polynomials $\{p_n\}$ will have a local uniform  estimate
$$
|p_n(x)|^2 \le C b_n^{-1}\,.
$$
\end{theorem}

Here is the theorem from \cite{Szw2003}. 
\begin{theorem}
\label{RSthm} Assume the sequences $a_n$ and $b_n$ satisfy $a_n \to \infty$, $b_n\to \infty$, 
$b_n/b_{n-1}\to 1$  and $a_n^2/b_n b_{n-1} \to 1/ 4B^2>1/4$.
Let the sequences
$$
\frac{a_n^2}{b_n b_{n-1}}, \,\, \frac{(b_n+ b_{n-1})}{ a_n^2},\,\, \frac1{a_n^2}
$$
have bounded variation. Then the corresponding Jacobi matrix $J$ with $b_n$ on the main diagonal is essentially
self-adjoint if and only if   $\sum a_n^{-1} =\infty$. In that case the spectrum of $J$ coincides 
with the whole real line and the spectral measure is absolutely continuous.
\end{theorem}

Theorem \ref{bn} follows from  this claim (except for the estimates of polynomials) immediately as the monotonicity of $\{ b_n\}$ ensures all the regularity required in Theorem \ref{RSthm}, and, of course, in assumptions of Theorem \ref{bn} $\sum b_n^{-1}=\infty$ gives $\sum a_n ^{-1}=\infty$.

Let us follow \cite{Szw2003} to show the estimate on orthogonal polynomials with respect to the spectral measure of $J$. There are several non essential typos in \cite{Szw2003}, and we will correct them on the way.

We have
\begin{equation}
\label{3terms}
xp_n(x)= a_{n+1} p_{n+1}(x) + b_n p_n(x) + a_{n} p_{n-1}(x)\,.
\end{equation}
We put 
\begin{equation}
\label{defAn}
A_n(x):= p_n(x) \sqrt{b_n-x},\quad n\ge N,\,\, \Lambda_n := B\frac{a_n}{\sqrt{(b_n-x)(b_{n-1}-x)}}\,.
\end{equation}

With this notation \eqref{3terms} becomes
\begin{equation}
\label{3t}
0= \Lambda_{n+1} A_{n+1}(x) + B A_n(x) + \Lambda_{n} A_{n-1}(x)\,.
\end{equation}

By assumptions, $\Lambda_n\to \frac12$ and $B<1$. Moreover,
since
$$
B^2\Lambda_n^{-2} = \frac{b_n b_{n-1}}{a_n^2}- \frac{(b_n+ b_{n-1})}{ a_n^2}x+  \frac1{a_n^2} x^2,
$$
it is of bounded variation, and thus so is $\Lambda_n$.

\begin{theorem} (Mat\'e, Nevai, \cite{MaN1983} ) 
\label{MN}
Let $\Lambda_n(x)$ be a positive valued sequence whose terms depend continuously on $x \in [a,b]$. Let $A_n(x)$ be a real valued sequence of continuous functions satisfying
\eqref{3t}
for $n \ge  N$. Assume the sequence $\Lambda_n(x)$ has bounded variation and $\Lambda_n(x)\to \frac12$ for
$x \in [a,b]$. Let  $|B| < 1$. Then there is a  strictly positive function $f(x)$ continuous on $[a,b]$
    such that
\begin{equation}
\label{det}
A_n^2(x) - A_{n-1}(x)A_{n+1}(x) \to f(x)
\end{equation}
uniformly for $x \in [a,b]$. Moreover, there is a constant $c$ such that
\begin{equation}
\label{bound}
|A_n(x)| \le c
\end{equation}
for $n\ge 0$ and $x \in [a,b]$.
\end{theorem}

Clearly to prove Theorem \ref{bn} it is sufficient to use  this result of Mat\'e, Nevai. Indeed, \eqref{bound} obviously gives us the bound on $p_n(x)^2$ stated in Theorem \ref{bn}.
For the reader’s convenience and for making the paper self-contained we give a proof to Theorem \ref{MN}.

\bigskip

\begin{proof} To prove Theorem \ref{MN}, one first uses recurrent relation \eqref{3t} to write
$$
A_{n-1} = -\frac{\Lambda_{n+1}}{\Lambda_{n} }A_{n+1}(x) -\frac{ B}{\Lambda_{n}} A_n(x) 
$$
and, hence,
\begin{equation}
\label{alg1}
A_n^2-A_{n-1}A_{n+1} = A_n^2 +\frac{\Lambda_{n+1}}{\Lambda_{n}} A_{n+1}^2 + \frac{B}{\Lambda_{n}} A_n A_{n+1}
\end{equation}
This can be rewritten as follows
\begin{equation}
\label{1}
A_n^2-A_{n-1}A_{n+1} = \Big( A_{n} + \frac{B}{2\Lambda_{n}} A_{n+1}\Big)^2 + \Big(\frac{\Lambda_{n+1}}{\Lambda_{n}}  - \frac{B^2}{4\Lambda_{n-1}^2} \Big) A_{n+1}^2
\end{equation}
Now we combine that equality with the facts that $\Lambda_n\to \frac12$ and $B<1$, and this combination   implies the following estimate:
\begin{equation}
\label{1a}
A_{n+1}^2 \le C (A_n^2-A_{n-1}A_{n+1})\,.
\end{equation}
But we can also rewrite the equality
\eqref{alg1}  in another form:
\begin{equation}
\label{2}
A_n^2-A_{n-1}A_{n+1} = \frac{\Lambda_{n+1}}{\Lambda_{n}} \Big( A_{n+1} + \frac{B}{2\Lambda_{n+1}} A_{n}\Big)^2 + \Big(1 - \frac{B^2}{4\Lambda_{n}\Lambda_{n+1}} \Big) A_{n}^2
\end{equation}

This formula and the same two facts that $\Lambda_n\to \frac12$ and $B<1$   imply now the following estimate:
\begin{equation}
\label{2a}
A_{n}^2 \le C (A_n^2-A_{n-1}A_{n+1})\,.
\end{equation}
Let us also write 
$$
A_{n+2} = - \frac{B}{\Lambda_{n+2}}A_{n+1} -\frac{\Lambda_{n+1}}{\Lambda_{n+2}} A_{n}
$$
That  equality  together with  \eqref{alg1}  give us the following:
\begin{align}
\label{alg2}
&(A_{n+1}^2  - A_nA_{n+2} ) - (A_n^2 -A_{n-1}A_{n+1} )= \Big(1-\frac{\Lambda_{n+1}}{\Lambda_{n}}  \Big) A_{n+1}^2 +\notag
\\ 
&B \Big(\frac1{\Lambda_{n+2}} -\frac1{\Lambda_{n}}\Big) A_n A_{n+1} + \big( \frac{\Lambda_{n+1}}{\Lambda_{n+2}} -1\Big) A_n^2\,.
\end{align}

Denoting $\Delta_n:= A_n^2- A_{n-1}A_{n+1}$ we get from \eqref{1a}, \eqref{2a} and \eqref{alg2}:
$$
\Delta_n >0,\,A_n^2 + A_{n+1}^2 \le C\Delta_n,\quad |\Delta_{n+1}-\Delta_n| \le C\big( |\Lambda_{n+1} -\Lambda_n|+ |\Lambda_{n+1} -\Lambda_{n+2}|\big)\Delta_n\,.
$$
Denote $\eps_n:=  |\Lambda_{n+1} -\Lambda_n|+ |\Lambda_{n+1} -\Lambda_{n+2}|$. Then
$$
(1-C\eps_n) \Delta_n \le \Delta_{n+1}\le (1+ C\eps_n) \Delta_n,
$$
and $\sum\eps_n$ converges by the assumption that $\Lambda_n$ has bounded variation.

Therefore, $\Delta_n$ uniformly converges to a strictly positive function $f$, and hence,
$A_n^2$ are uniformly bounded uniformly bounded for $n > N(x)$ (namely, by a multiple $Cf(x)$ of $f(x)$). Thus \eqref{bound} is proved, Theorem \ref{MN}  of Mat\'e--Nevai is proved, and we already said that this proves the bound of Theorem \ref{bn}.

\end{proof}
%%%%%%%%%%%%%%%%%%%%

\section{Counterexamples: an attempt on Bessel systems, 
and proof of Theorem 1.2.}
\label{4}
\subsection{Part I of the Theorem 1.2.}
\label{4.1}
  \rm A \it natural question whether a ``half of the frame 
condition"\rm , namely the Bessel one, \it is sufficient \rm for getting 
the conclusion of theorem 1.1, is essentially equivalent (in the notation 
of Theorem 1.2) to the following: whether\, 
$$
(1)\&(2)  \Rightarrow  (3')\it :=  \exists f\in L^{2}_{{\Bbb R}}(\Omega 
,\mu ) \rm \,\,\text{such that}\,\,\sum _{n}\vert (f,v_{n})\vert ^{2}=\, \infty 
\, ?
$$
\, \, \rm Indeed, if (3') does not hold (and we have $ \sum _{n}\vert 
(f,v_{n})\vert ^{2}<\, \infty $, $ \forall f\in L^{2}_{{\Bbb R}}(\Omega 
,\mu )$), we automatically get property (3) of theorem 1.3 just due 
to Banach--Steinhaus theorem applied to the semi-norms\, 
$$
 p_{n}(f)=\, \Big (\displaystyle \sum _{k=1}^{n}(f,v_{k})^{2}\Big )^{1/2}\,.
 $$
\, 
\, \, \rm However, there is a counterexample which gives 
a negative answer to this question and proves Part I of the Theorem 1.2.\, 

\medskip

\subsection{Counterexample}
\label{counter}
 \rm Let $ (\Omega ,\mu )=\, (0,1),\, 
dx$, and $ (v_{k})_{k\geq 1}$ be any enumeration of the indicator functions 
$ \chi _{I}$ of dyadic subintervals $ {\mathcal D}=\, \{I=\, I_{j,n}\}\, 
$of $ (0,1)$:\, 
$$
 I_{j,n}=\, ({\frac{\displaystyle j}{\displaystyle 2^{n}}} 
,{\frac{\displaystyle j+1}{\displaystyle 2^{n}}} ),\, j=\, 0,...,\, 
2^{n}-1\,.
$$

\medskip

\, \rm Properties (1) and (2), as well as the completeness of 
$ (v_{k})$, are obvious. For (3), we write\, 
$$
 \displaystyle \sum _{k}\Big \vert (f,v_{k})\Big \vert ^{2}=\, 
\displaystyle \sum _{I\in {\mathcal D}}\Big ({\frac{\displaystyle 1}{\displaystyle \vert 
I\vert }} \displaystyle \int _{I}fdx\Big )^{2}\Big \vert I\Big \vert ^{2}\,,
$$
\rm and notice that the desired property (3) is the ``Carleson embedding"
$$
 \displaystyle \sum _{I\in {\mathcal D}}\Big ({\frac{\displaystyle 1}{\displaystyle 
\vert I\vert }} \displaystyle \int _{I}fdx\Big )^{2}w_{I}\leq \, B\Big \Vert 
f\Big \Vert _{2}^{2}\,,
$$
where $ w_{I}=\, \vert I\vert ^{2}$, $ I\in {\mathcal D}$. The 
necessary and sufficient condition for such an embedding is (see [NTV1999], 
[NT1996])\, 
$$
 \sup_{J\in {\mathcal D}}{\frac{\displaystyle 1}{\displaystyle \vert 
J\vert }} \displaystyle \sum _{I\subset J,I\in {\mathcal D}}w_{I}<\, \infty 
\,,
$$ 
\rm which is obviously fulfilled for $ w_{I}=$ $ \vert I\vert ^{2}$, 
$ I\in {\mathcal D}$.

\bigskip

\subsection{Part II of the Theorem 1.2}
\label{4.3}
 \rm Take $ \Omega =\, (0,2)$, and let $ (v_{n})_{n\geq 
1}$ be the sequence in $ L^{2}_{{\Bbb R}}((0,1),dx)$ constructed in 
Part I. Without loss of generality, we suppose that $ B<\, 1$. 
Then, the linear mapping $ T:l^{2}\longrightarrow L^{2}(0,1)$ defined 
by $ T\delta _{n}=\, v_{n},\, n\geq 1$ ($ \delta _{n}$ stands 
for the natural basis of $ l^{2}$) is a (strict) contraction. Let $ 
D_{T}=\, (I-T^{*}T)^{1/2}:l^{2}\longrightarrow l^{2}$ its defect 
operator, and $ V:l^{2}\longrightarrow L^{2}(1,2)$ an arbitrary isometric 
operator. We naturally consider $ L^{2}(0,2)$ as an orthogonal sum 
$ L^{2}(0,2)=\, L^{2}(0,1)\oplus L^{2}(1,2)$ and set $ Ux=\, Tx\oplus 
VD_{T}x$ for $ x\in l^{2}$. Then, $ U$ is isometric, $ \Vert Ux\Vert ^{2}=\, 
\Vert Tx\Vert ^{2}+\, \Vert D_{T}x\Vert ^{2}=\, \Vert x\Vert ^{2}$, 
and hence $ u_{2n}:=\, U\delta _{n}$, $ n=1,2,...$ is an orthonormal 
basis in $ F:=\, Ul^{2}\subset \, L^{2}(0,2)$. Choosing an 
arbitrary orthonormal basis $ (u_{2n+1})_{n\geq 1}$ in the orthogonal 
complement $ F^{\perp }$, we obtain an orthonormal basis $ (u_{k})_{k\geq 
1}$ in $ L^{2}(0,2)$ satisfying all requirements of the theorem (with 
$ E=\, (0,1)$).

\subsection{A lapse of equidistribution between $ u_{k}^{\pm }(x)$}
\label{4.4}
\begin{proof}
 \rm One 
can reordering the basis from 4.2 in order to get the following: \it there 
exists an orthonormal basis $ (U_{k})$ in $ L_{\bR}^{2}(0,2)$ such that\, 
$$
 \displaystyle \sum _{k=1}^{n}(U_{k}^{-}(x))^{2}=\, o\Big (\displaystyle \sum 
_{k=1}^{n}(U_{k}^{+}(x))^{2}\Big )\,\,\text{as}\,\, n\longrightarrow \infty \,\,
 x\in (0,1)\,.
 $$
Indeed, it suffices to set\, 
$$
(U_{k}):\, u_{2}, u_{4},...,u_{2N_{1}},\,\, u_{1},\,\,
u_{2N_{1}  +2}, ...,\, u_{2N_{2}}{\rm ,}u_{3},...
$$ 
where the integers $ N_{1}<\, N_{2}<...$ increase sufficiently 
fast. 
\end{proof}
\, 
\subsection{A minimal sequence can be positive}
\label{4.5}
 \rm Let $ u_{k}(x)=\, {\frac{1}{1+\sqrt{2} 
}} (1+\, Cos(\pi kx))$, $ k=\, 1,2,...$ in $ L_{\bR}^{2}(0,1)$. 
Then $ (u_{k})$ spans $ L_{\bR}^{2}(0,1)$, is normalized and uniformly 
minimal (with the dual $ u'_{k}=\, (2+\sqrt{2} )Cos(\pi kx)$), 
and $ u_{k}(x)\geq 0$. In fact, the Fourier series with respect to 
$ (u_{k})$ of a function $ f\in L_{\bR}^{2}(0,1)$, $ \sum _{k\geq 1}(f,u'_{k})u_{k}$, 
converges to $ f$, if $ f$ is (for example) Dini continuous at $ x=0$ 
and $ f(0)=\, 0$. However, $ (u_{k})$ is not a basis.

The question of the existence of non-negative Schauder basis in $L^p, p>1$ is open to the best of our knowledge. Detailed discussion can be found in \cite{PS2016}. For $p=1$, as it is already mentioned, non-negative Schauder basis exists, see \cite{JS2015}.

\end{document}